\documentclass[12pt,a4paper,draft,twoside,reqno]{amsart}
\usepackage{amsfonts,amssymb,euscript,amscd}
\DeclareFontFamily{OML}{cyr}{} \DeclareFontShape{OML}{cyr}{m}{n}{
  <5> <6> <7> <8> <9> gen * wncyr
  <10> <10.95> <12> <14.4> <17.28> <20.74> <24.88> wncyr10
  }{}
\DeclareSymbolFont{rusletters}{OML}{cyr}{m}{n}
\DeclareSymbolFontAlphabet{\rusmath}{rusletters}
\DeclareMathSymbol\re{\rusmath}{rusletters}{"03}
\newtheorem{theorem}{Theorem}[] 

\newtheorem{lemma}[theorem]{Lemma}
\newtheorem{corollary}[theorem]{Corollary}

\newtheorem{example}[theorem]{Example}
\newcommand*{\R}{\mathbb R} 
\newcommand*{\1}{\mathop{\bf 1}\nolimits}

\begin{document}
\title[Linear differential operators of forth order]{Invariants of forth order linear differential operators}
\author[Valentin Lychagin]{Valentin Lychagin}
\address{University of Tromso, Tromso, Norway; Institute of Control Sciences of RAS, Moscow, Russia}
\email{Valentin.Lychagin@matnat.uit.no}
\author[Valeriy Yumaguzhin]{Valeriy Yumaguzhin}
\address{
Program Systems Institute of RAS, Pereslavl'-Zales\-skiy, Russia;  Institute of Control Sciences of RAS, Moscow, Russia}
\email{yuma@diffiety.botik.ru}
\thanks{V. Yumaguzhin is a corresponding author; phone: +79056368327, e-mail: yuma@diffiety.botik.ru }

\begin{abstract}
In this paper, we study scalar the forth order linear differential operators  over an oriented 2-dimensional manifold. We investigate differential invariants of these operators and show their application to the equivalence problem. 
\end{abstract}

\maketitle

\section{ Introduction}
In the papers \cite{LY2, LY3, LYk}, we analyzed the equivalence of $k$-order linear differential operators acting on a line bundle over smooth $n$-dimensional manifold. In \cite {LY2}, we investigated the case of arbitrary $ n $ and $ k = 2 $, in \cite {LY3}, we investigated the case of $ n = 2 $ and $ k = 3 $, and in \cite {LYk}, we investigated the case $ n \ge 2 $ and $ k \ge 3 $ for constant type operators.

It is shown in \cite{LYk} that a stationary Lie algebra of symbol of regular operator is trivial at every point when $n \geq 2$ and $k \geq 3$ and therefore codimension of regular orbit of this symbol is 
\begin{equation}\label{CdmZrOrdOrb} 
c (n, k) =\binom{n+k-1}k - n^2.
\end{equation}

It is easy to check that $c (n, k)\geq n$ for all $n\geq 2$ and $k \geq3$, with the exception for the following three cases: 
\begin{gather*}
  n = 2, k = 3;\\ 
  n = 2, k = 4;\\ 
  n = 3, k = 3.
\end{gather*}
It follows, see \cite{LYk}, that the field of natural invariants of regular operator is generated by zero order invariants  in non exceptional cases.

The case $n=2$, $k=3$ was investigated in \cite{LY3}.   

 In this paper, we consider the case, $n=2$, $k=4$. In this case $c (2, 4)=1$. This means that there is a unique independent differential invariant of zero order for the regular symbols.  

Essentially, this invariant has long been known. Indeed, in the Hilbert lectures \cite{Hlbrt}, p. 57, two relative invariants for a forth degree homogeneous polynomial of two variables
$a_0x^4+4\,a_1x^3y+6\,a_2x^2y^2+4\,a_3x\,y^3+a_4y^4$  are found:
\begin{align*}
  &{\EuScript I}_2=a_0a_4-4a_1a_3+3a_2^2,\\
  &{\EuScript I}_3=a_0a_2a_4-a_0a_3^2-a_1^2a_4+2a_1a_2a_3-a_2^3.
\end{align*}
These relative invariants have weights 2 and 3 respectively. Hence,  
\begin{equation}\label{Inv_0}
   I_0={\EuScript I}_3^2/{\EuScript I}_2^3
\end{equation}
is an invariant for forth degree homogeneous polynomials in two variables. 

Thus $I_0$ is a rational invariant of a principal symbol of a forth order linear differential operator on two dimensional manifolds and hence $I_0$ is a zero order rational differential invariant of this operator.

We use this invariant in the following way.

If $I_0$ is not constant, then it generates the second differential invariant of 1-st order.  Using   both of these invariants as natural coordinates, we get the full classification of  4th order operators with the non constant invariant $I_0$ w.r.t. the group of diffeomorphism of the manifold.

If $I_0$ is constant, then we have the case of the constant type 4th order operators which is considered in \cite{LYk}. For such operators, there exist the Wagner connections in the manifold.
Using this connection, we reduce the classification problem for the constant type 4th order operators to the classification problem for set of four symmetric tensors of orders 4, 3, 2, 1, 0. This allows us to find the differential invariants for such kind of operators and solve the equivalence problem for them.

\subsection{Notations}

In this paper, we use the same notations as  in \cite{LY2,LY3,LYk}.

Let $M$ be $n$-dimensional manifold. 

Let $\tau : TM\to M$ and $\tau^* : T^*M \to M$ be respectively tangent and cotangent bundles over $M$, let $\Sigma_k (M) = C^{\infty}\big(S^k (\tau)\big)$ be  the module of symmetric $k$-vectors, $\Sigma^k (M) =C^{\infty}\big(S^k (\tau^*)\big)$  the module of symmetric $k$-forms,  $\Omega_k (M) =C^{\infty}\big(\Lambda^k(\tau)\big)$  the module of skew-symmet\-ric $k$-vectors, and $\Omega^k (M) =C^{\infty}\big(\Lambda^k(\tau^*)\big)$ the module of exterior $k$-forms. 

We denote by $\mathbf{Diff}_k (\1)$ a left $C^{\infty}(M)$- module of linear differential operators of order $\leq k$, acting on the trivial line bundle $\1:M\times\R\to M$, and by $\chi_k\! :{\rm Diff}_k (\1)\to M$, we denote the bundle of these differential operators, thus $C^{\infty}(\chi_k) = \mathbf {Diff}_k(\1)$.  

By $\mathcal{F}(M)$ we denote a multiplicative group of smooth functions on M without zeros on M, by $\mathcal{G}(M)$ will be denoted a group of diffeomorphisms of $M$. 

By ${\rm GL}(V)$ we denote a group of all linear transformations of vector space $V$.

By principal symbol $\sigma_k(A)$ of operator $A \in\mathbf {Diff}_k(\1)$ we mean the equivalence class
$$
\sigma_k(A)\equiv A\!\!\!\mod \mathbf{Diff}_{k-1}(\1)\in\Sigma_k(M).
$$

\section{Symbols of differential operators on 2-dimensional manifolds}

Let $M$ be an oriented 2-dimensional manifold and $x$, $y$ be local coordinates in $M$. 

An operator $A\in \mathbf {Diff}_4 (\1)$ has the following form in the coordinates $x,y$,
\begin{gather*}
   A=a_0\partial_x^4+4a_1\partial_x^3\partial_y+6a_2\partial_x^2\partial_y^2
+4a_3\partial_x\partial_y^3+a_4\partial_y^4\\
 +b_0\partial_x^3+3b_1\partial_x^2\partial_y+3b_2\partial_x\partial_y^2+b_3\partial_y^3
+c_0\partial_x^2+2c_1\partial_x\partial_y+c_2\partial_y^2\\
+d_0\partial_x+d_1\partial_y +e_0.
\end{gather*}

Its principal symbol $\sigma_4(A)\in\Sigma_4(M)$ is the following symmetric 4-vector in the coordinates $x,y$,
\begin{equation}\label{sgm}
   \sigma_4(A) =a_0\partial_x^4+4a_1\partial_x^3\cdot\partial_y+6a_2\partial_x^2\cdot\partial_y^2
+4a_3\partial_x\cdot\partial_y^3+a_4\partial_y^4,
\end{equation}
where we denoted by $\cdot$ the symmetric product and by $\partial_a^k$ the symmetric product of $k$ copies of $\partial_a$. 

In canonical coordinates $x, y,  p_x, p_y$ on $T^*M$ this tensor is a forth degree homogeneous  polynomial (Hamiltonian)
$$ 
  \sigma_4(A) = a_0p_x^4+4a_1p_x^3p_y+6a_2p_x^2p_y^2+4a_3p_xp_y^3+a4p_y^4
$$
on $T^*M$. 

There are three possibilities for roots of this polynomial: all roots are real, two roots are real and other are complex, and all roots are complex.

Let $\EuScript{D}(A)$ be a discriminant of $\sigma_4(A)$.  Then $\EuScript{D}(A)=0$ if and only if the symbol $\sigma_4(A)$ has multiple roots, or characteristics 

We say that the symbol $\sigma_4(A)$ is {\it regular} if $\EuScript{D}(A)\neq 0$. 

In this paper, we consider only regular operators, i.e., operators with regular symbols. 

One can check that: $\EuScript{D}(A)>0$ if and only if all roots of $\sigma_4(A)$ are distinct and real or all roots 
are distinct and complex, 

$\EuScript{D}(A)<0$ if and only if two roots of $\sigma_4(A)$ are distinct and real and other two are complex. 

\begin{lemma}$\phantom{\sigma}$
 \begin{enumerate}
  \item If symbol $\sigma_4(A)$ has two distinct real roots then there are local coordinates $x, y$ such that 
\begin{equation}\label{4RlRts}
 \sigma_4(A)=\partial_x\cdot\partial_y\cdot\big(\alpha_0\partial_x^2+2\alpha_1\partial_x\cdot\partial_y+\alpha_2\partial_y^2\big).
\end{equation}
 \item If symbol $\sigma_4(A)$ has complex root then there are local coordinates $x, y$, which we call isothermal coordinates, such that 
\begin{equation}\label{4CmplRts}
 \sigma_4(A)=(\partial_x^2+\partial_y^2)\cdot\big(\alpha_0\partial_x^2+2\alpha_1\partial_x\cdot\partial_y+\alpha_2\partial_y^2\big).
\end{equation}
\end{enumerate}
\end{lemma}

Remark that if $z=x+iy$, $\bar z=x-iy$ are complex coordinates in a domain of the isothermal coordinates, then symbol \eqref {4CmplRts} has form
\begin{equation}\label{4CmplRtsZ}
 \sigma_4(A)=\partial_z\cdot\partial_{\bar z}\cdot\big(\alpha_0\partial_z^2+2\alpha_1\partial_z\cdot\partial_{\bar z}+\alpha_2\partial_{\bar z}^2\big).
\end{equation}
 similar to \eqref{4RlRts}.
\begin{corollary}$\phantom{\sigma}$
 \begin{enumerate}
  \item If symbol $\sigma_4(A)$ is defined by \eqref{4RlRts} or \eqref{4CmplRtsZ}, then its discriminant is defined up to a positive numerical factor by the formula
$$ 
  \EuScript{D}(A)=
   \alpha_0^2\alpha_2^2\,(9\,\alpha_1^2-16\,\alpha_0\alpha_2).
$$ 
  \item If symbol $\sigma_4(A)$ is defined by \eqref{4CmplRts}, then its discriminant is defined up to a positive numerical factor by the formula 
$$ 
  \EuScript{D}(A)=
  \big(4\,\alpha_1^2+(\alpha_0-\alpha_2)^2\big)^2(\alpha_0\alpha_2-\alpha_1^2).
$$ 
\end{enumerate}      
\end{corollary}

\section{Invariants of non constant type differential operators}
\subsection{Constant type operators}
Let $V$ be a 2-dimensional vector space and let $\varpi\subset S^k(V)$ be a regular $\mathrm{ GL}(V)$-orbit (i.e., $\varpi$ is defined  by equations  $I_1=c_1,\ldots, I_m=c_m$, where $I_i$ are independent $\mathrm{GL}(V)$-invariants  in a neighborhood of $\varpi$, $c_i$ are constant, and $m$ is codimension of $\varpi$). 

Recall, see \cite{LYk}, that:
\begin{enumerate}
  \item a symbol $\sigma\in\Sigma_4(M)$ has a {\it constant type $\varpi$} if for any point $q\in M$ and any isomorphism $\varphi : T_q(M)\to V$ the image  $\varphi_* (\sigma)\in S^k(V)$ belongs to $\varpi$; 
\item an operator $A\in {\bf Diff}_4(\1)$ has the {\it constant type $\varpi$} if its symbol $\sigma_4(A)$ has the constant type $\varpi$.
\end{enumerate}
Remark that a symbol $\sigma\in\Sigma_4(M)$ has a constant type if and only if its zero order rational differential invariant $I_0(\sigma)$ is a constant.

\subsection{The bundle of differential operators} 
In the bundle\linebreak $\chi_4\!:\mathrm {Diff}_4(\1)\to M$ we will use the following canonical local coordinates $(x, y, u^{\alpha})$,
where $(x, y)$ are local coordinates in $M$ and $u^{\alpha}$ are fiber wise coordinates in bundle  $\chi_4$. Here $\alpha = (\alpha_1, \alpha_2)$ is the multi index of length $0\le |\alpha|=\alpha_1+\alpha_2 \le 4$.

In these coordinates the section 
$$
S_A:M\longrightarrow \mathrm{Diff}_4(\1) 
$$ 
that corresponds to operator 
$$
   A =\sum_{|\alpha|\le 4}a^{\alpha}(x,y)\partial_x^{\alpha_1}\partial_y^{\alpha_2}\in\mathbf{Diff}_4(\1),
$$
has the form
$$
    u^{\alpha} = a^{\alpha}(x,\, y).
$$

Denote by $\pi_l : J^l(\chi_4) \to M$ the vector bundles of $l$-jets of sections of bundles $\chi_4$, or, in other words, bundles of $l$-jets of the $4$-order scalar differential operators.

We will denote by $[A]^l_p$ $l$-jets of operators at a point $p\in M$.

Bundles $\chi_4$, as well as bundles $\pi_l$ are natural in the sense that the action of the diffeomorphism group $\mathcal{G}(M)$ is lifted to automorphisms of these bundles in the natural way:
$$
\varphi^{(l)} : [A]^l_p \mapsto [\varphi_*(A)]^l_{\varphi(p)}
$$
for any diffeomorphism $\varphi\in\mathcal{G}(M)$.

The total differential operator of order $4$, see \cite{LYk}, 
$$
\square : C^{\infty}\big(J^l(\chi_4)\big)\longrightarrow C^{\infty}\big(J^{l+4}(\chi_4)\big),
\;l = 0, 1, \ldots, 
$$
is defined by the formula 
$$ 
j_{4+l}(S_A)^*\big(\square(f)\big) = A\big(j_l (S_A)^*(f)\big),
$$ 
for all functions $f\in C^{\infty}\big(J^l(\chi_4)\big)$ and operators $A \in\mathbf{Diff}_4(\1)$.
It is easy to check that in the standard jet-coordinates in the bundles $\pi_l$ this operator has the
 following form
$$
\square =\sum_{|\alpha|\le 4}u^{\alpha}\frac{d^{|\alpha|}}{d_x^{\alpha_1}d_y^{\alpha_2}}.
$$

The main property of this operator is its naturality:
$$
\varphi^{(4+l)*} \circ\square = \square\circ \varphi^{(l)*}
$$
for all diffeomorphisms $\varphi\in\mathcal{G}(M)$.

Let $J = (J_1, J_2)$ be a pair of natural differential invariants. We say that they are {\it in general position} if
$$
\widehat{dJ_1}\wedge\widehat{dJ_2} \neq 0.
$$
Let $\EuScript I$ be an invariant, then
$$
\widehat{d\EuScript I} = \EuScript I_1\widehat{dJ_1} +\EuScript I_2\widehat{dJ_2},
$$
for some rational functions $\EuScript I_i$, which are called {\it Tresse derivatives}.
We will denote them by $\displaystyle\frac{d\EuScript I}{dJ_i}$. They are invariants by the construction, having, as a rule, higher order then the invariant $\EuScript I$.

\subsection{}
Let $A\in\mathbf{Diff}_4(\1)$ be a non constant type regular operator, \linebreak
$\sigma_4(A)\in\Sigma_4(M)$ be its symbol, and $I_0(A)$ be its zero order rational differential invariant defined in local coordinates $x, y$ of $M$  by \eqref{Inv_0},
$$ 
   I_0(A)={\EuScript I}_3^2/{\EuScript I}_2^3,
$$ 
where 
$$
\EuScript{I}_2=a_0a_4-4a_1a_3+3a_2^2,\quad
  \EuScript{I}_3=a_0a_2a_4-a_0a_3^2-a_1^2a_4+2a_1a_2a_3-a_2^3,
$$ 
and $a_0,\ldots, a_4$ are coefficients of $\sigma_4(A)$, see \eqref{sgm}.

Assume that
$$ 
  dI_0(A) \neq 0
$$  
in some open domain of $M$.

Then the convolution $\big\langle\, dI_0(A)^4,\,\sigma(A)\,\big\rangle$ of symmetric differential 4-form $dI_0(A)^4$ and  symbol $\sigma_4(A)$ is a first order rational differential invariant of $A$. We denote it by $I_1(A)$, 
\begin{equation}\label{Inv_1}
 I_1(A)=\big\langle\, dI_0(A)^4,\,\sigma(A)\,\big\rangle.
\end{equation} 

\subsection{The field  of all natural rational differential invariants of non constant type operators} 

We say that 2-jet $\theta_2\in J^2\chi_4$ is {\it regular} if $\widehat{d I_0}\wedge\widehat {dI_1}\neq 0$ at the point $\theta_2$. Denote by $\EuScript{S}_2\subset J^2\chi_4$ the set of all singular points of $J^2\chi_4$, i.e., 
$$
 \EuScript{S}_2 =\{\theta_2\in J^2\chi_4\,|\,(\widehat{dI_0}\wedge\widehat{dI_1})_{\theta_2}= 0\}.
$$

In the regular domain $J^2\chi_4\setminus\EuScript{S}_2$ the invariants $I_0$ and $I_1$ are in general position.

From \cite{LYk}, Theorem 11, we get the following statement .

\begin{theorem} The field of all natural rational invariants of non constant operators $A\in\mathbf{Diff}_4(\1)$ is generated by the invariants $I_0, I_1$, and the Tresse derivatives
$$
\frac{d^{|\beta|}J_{\alpha}}{dI_0^{\beta_1}dI_1^{\beta_2}}
$$
of invariants
$$
 J_{\alpha} =\square(I_0^{\alpha_1}\cdot I_1^{\alpha_2})
$$
with $0 \le |\alpha| \le 4$.

The field of rational natural invariants separates, see \cite{KL, Ros}, regular orbits in the jet spaces of differential operators of non constant type.
\end{theorem}

\section{Invariants of constant type differential operators}
\subsection{The Wagner connection}
\begin{theorem} A regular symbol $\sigma\in\Sigma_4$ has the constant type if and only if it has a linear connection $\nabla^{\sigma}$ in the tangent bundle to $M$ such that 
\begin{equation}\label{WgnrCnnct}
 \nabla^{\sigma}_X(\sigma)=0
\end{equation}
for all vector fields $X$ on $M$. 
\end{theorem}
\begin{proof} 
Suppose that $\sigma\in\Sigma_4(M)$ is regular and has a constant type $\varpi\subset \Sigma_4(M)$.   
Then, (see \cite{LYk},  Corollary 9), for any point $p\in M$ there are a neighborhood ${\EuScript O}_p$ and a unique linear isomorphisms $A_{p,p'} : T_p M \to T_{p'} M$, for all $p' \in {\EuScript O}_p$, such that $\big(A_{p,p'}\big)_*(\sigma_{p'})= \sigma_p$. It follows that there are a unique linear isomorphisms $A_{p,p'} : T_p M \to T_{p'} M$, for all $p' \in M$, such that $\big(A_{p,p'}\big)_*(\sigma_{p'})= \sigma_p$. Therefore, there is a unique linear connection $\nabla^{\sigma}$ on manifold $M$ such that $\nabla^{\sigma}_X(\sigma)=0$ for all vector fields $X$ on $M$. We call it Wagner connection, see \cite{ LY3}.

Inversely, let $\sigma\in\Sigma_4(M)$ and let  $\nabla^{\sigma}$ be its Wagner connection, and $p\in M$. By $\{(e_1) _p,  (e_2) _p\}$ we denote a base in the tangent space $T_p (M)$, then transferring each vector $(e_i)_p$ in parallel along the Wagner connection at every point in $M$, we get the frame $\{e_1,  e_2\}$ on $M$. In the terms of this frame, 
$\sigma=a_0e_1^4+4a_1e_1^3\cdot e_2+6a_2e_1^2\cdot e_2^2+4a_3e_1^1\cdot e_2^3+a_4e_2^4$, where $a_i\in\R$. This means that $\sigma$ has a constant type.
\end{proof}

From construction of Wagner connection, we get
\begin{corollary}
The curvature tensor of the Wagner connection $\nabla^{\sigma}$ is equal to zero. 
\end{corollary}

\subsubsection{Coordinates}

Suppose that the invariant $I_0$ of $\sigma$ is a constant. 

Here we express Christoffel symbols $\Gamma^i_{jk}$  of the Wagner connection $\nabla^{\sigma}$ in terms of coefficients of $\sigma$.

Let $\partial_1=\partial_x$, $\partial_2=\partial_y$. 
Then the symbol $\sigma$ can be rewritten in the form
 $$
  \sigma=a^{i_1\ldots i_4}\partial_{i_1}\cdot\ldots\cdot\partial_{i_4},
$$ 
where is the summation over repeated indices, $i_1,\ldots,i_4=1,2$, and coefficients $a^{i_1\ldots i_4}$ are symmetric in superscripts. Now condition  \eqref{WgnrCnnct} is the following
\begin{equation}\label{EqWgnrCnnct}
  \nabla^{\sigma}_{\partial_l}\,a^{i_1\ldots i_4}\!=\!\partial_l\,a^{i_1\ldots i_4}+\Gamma^{i_1}_{ml}a^{mi_2i_3i_4}+\ldots+\Gamma^{i_4}_{ml}a^{i_1i_2i_3m}\!=\!0,\; l\!=\! 1,2.
\end{equation}

System \eqref{EqWgnrCnnct} consist of 10 linear algebraic equations on 8 unknown functions $\Gamma^i_{jk}$. 

Let us fix some component $a^{k_1 k_2 k_3 k_4}$ of $\sigma$, for example $a^{1 1 2 2}$. Excluding  two equations $\partial_1a^{k_1 k_2 k_3 k_4}+\ldots=0$ and  $\partial_2a^{k_1 k_2 k_3 k_4}+\ldots=0$ from system \eqref{EqWgnrCnnct}, we get the system of 8 equations. 

One can check that the determinant of this system is proportional to the discriminant $\EuScript{D}(\sigma)$ of the symbol $\sigma$ and
therefore this system has a unique solution. 

Moreover, one can check that this solution is independent of choice of the component $a^{k_1 k_2 k_3 k_4}$.

The solution satisfies also to the excluded equations $\partial_1a^{k_1 k_2 k_3 k_4}+\ldots=0$ and  $\partial_2a^{k_1 k_2 k_3 k_4}+\ldots=0$. Indeed, one can check that substituting the solution in the excluded equations, we get two expressions which are $\partial_1 I_0$ and $\partial_2 I_0$ respectively. 

\begin{example} Let regular symbol $\sigma\in\Sigma_4(M)$ be defined in local coordinate by formula \eqref{4RlRts},
$$ 
 \sigma=\partial_x\cdot\partial_y\cdot\big(\alpha_0\partial_x^2+2\alpha_1\partial_x\cdot\partial_y+\alpha_2\partial_y^2\big).
$$ 
Then non zero components $\Gamma^i_{jk}$ of its Wagner connection $\nabla^{\sigma}$ are defined in these coordinates by the formulas:
\begin{align*}
  \Gamma^1_{1\,1}&=(-3\alpha_2\partial_x\alpha_0+\alpha_0\partial_x\alpha_2)/(8\alpha_0\alpha_2),\\
\Gamma^1_{2\,1}&=(-3\alpha_2\partial_y\alpha_0+\alpha_0\partial_y\alpha_2)/(8\alpha_0\alpha_2),\\
\Gamma^2_{1\,2}&=(\alpha_2\partial_x\alpha_0-3\alpha_0\partial_x\alpha_2)/(8\alpha_0\alpha_2),\\
\Gamma^2_{2\,2}&=(\alpha_2\partial_y\alpha_0-3\alpha_0\partial_y\alpha_2)/(8\alpha_0\alpha_2).
\end{align*}
\end{example}

\subsection{Group-type symbols}
Let $M$ be a connected and simply connected manifold, $\sigma$ be a regular symmetric 4-vector, and $\nabla^{\sigma}$ the Wagner connection.  Assume that this connection is complete.

We assume also that the torsion tensor $T^{\sigma}$ of Wagner connection $\nabla^{\sigma}$ is parallel, i.e., 
$$ 
d_{\nabla^{\sigma}} (T^{\sigma}) = 0. 
$$
Then, it is easy to check that the 2-dimensional vector space $\mathfrak g^{\sigma}$ of all parallel vector fields on $M$, is a Lie algebra with respect to the bracket
$$
X, Y \in \mathfrak g^{\sigma}\longrightarrow T^{\sigma}(X, Y)\in\mathfrak g^{\sigma}.
$$

For this 2-dimensional Lie algebra $\mathfrak g^{\sigma}$, we have:
\begin{enumerate}
\item $T^{\sigma}(X,Y)=[X,Y]$ for all $X,Y\in \mathfrak g^{\sigma}$, here $[X,Y]$ is the usual bracket of vector fields,
\item algebra $\mathfrak g^{\sigma}$  is either commutative or solvable,
\item algebra $\mathfrak g^{\sigma}$ is commutative if and only if $T^{\sigma}=0$, and
\item algebra $\mathfrak g^{\sigma}$ is solvable if and only if $T^{\sigma}\neq 0$,
\item if algebra $\mathfrak g^{\sigma}$ is solvable, then there is basis $X, Y\in\mathfrak g^{\sigma}$  such that  $[X, Y]=X$. 
\end{enumerate}

\begin{theorem}
 Let $\sigma\in\Sigma_4(M)$ be a regular symbol and let $\nabla^{\sigma}$ be the corresponding Wagner connection with parallel torsion tensor $T^{\sigma}$. Then:
\begin{enumerate}
\item Symbol $\sigma$ is locally equivalent to the symbol with constant coefficients
$$
   \sigma= c_0\partial_x^4+ 4c_1\partial_x^3\cdot\partial_y + 6c_2\partial_x^2\cdot\partial_y^2
+4c_3\partial_x^3\cdot\partial_y+ c_4\partial_y^4,\quad c_i\in\R,
$$
if and only if $T^{\sigma}=0$.
\item Symbol  $\sigma$ is locally equivalent to the symbol 
\begin{gather*}
 \sigma= c_0e^{4y}\,\partial_x^4+ 4c_1e^{3y}\,\partial_x^3\!\cdot\!\partial_y + 6c_2e^{2y}\,\partial_x^2\!\cdot\!\partial_y^2
+4c_3e^y\,\partial_x\!\cdot\!\partial_y^3+ c_4\partial_y^4,\\
c_i\in\R,
\end{gather*}
if and only if $T^{\sigma}\neq 0$.
\end{enumerate}
\end{theorem}
\begin{proof} {\rm(1)} The condition $T^{\sigma}=0$ means that [X,Y]=0 for any parallel vector fields $X$ and $Y$ on $M$. 

Let vector fields $X, Y$ be parallel and linearly independent at every point.  Then there exist local coordinates $x,y$ in $M$ such that $X=\partial_x$ and $Y=\partial_y$. 
 
It follows that all components $\Gamma^i_{jk}$ of the Wagner connection $\nabla^{\sigma}$ 
are equal to zero in domain of these coordinates $x,y$.

From equations \eqref{EqWgnrCnnct}, we get that coefficients of $\sigma$ in the coordinates $x,y$ are constants.

{\rm(2)} Let  $T^{\sigma}\neq 0$ and $d_{\nabla^{\sigma}}(T^{\sigma}) = 0$. 
Let vector fields $X, Y$ be a basis in the algebra $\mathfrak g^{\sigma}$ and $[X,Y]=X$. Then  there are local coordinates $x,y$ in $M$ such that $X=\partial_x$ and $Y=\alpha(x,y)\partial_x+\beta(x,y)\partial_y$.
From the condition $[X,Y]=X$, we get that $\alpha(x,y)=x+c$, $c\in\R$, and $\beta$ is independent of $x$. Linear independence of $X$ and $Y$ means that $\beta\neq 0$ everywhere. Thus, we can take $Y=x\partial_x+\partial_y$.

The parallelism of each vector field $X=\partial_x$ and $Y=x\partial_x+\partial_y$ gives us that 
$\Gamma^1_{21}=-1$ and all other components $\Gamma^i_{jk}$ are equal to zero. 

Now, system \eqref{EqWgnrCnnct} is the following,
$$ 
   \partial_x a_i=0,\;\;
  \partial_y a_i-(4-i)a_i=0,\quad  i=0,1,2,3,4.
$$ 
Whence it follows that 
$$ 
   a_0=c_0e^{4y},\; a_1=c_1e^{3y},\; a_2=c_2e^{2y},\; a_3=c_3e^y,\; a_4=c_4,\quad c_i\in\R.
$$ 
\end{proof}

\subsection{Symbols and quantization} 
Let $\Sigma^{\cdot} = \oplus_{k\ge 0}\Sigma^k(M)$ be the gra\-ded algebra of symmetric differential forms and let $\nabla$ be a Wagner connection of a regular symbol from $\Sigma_4(M)$. Then the covariant differential
$$
d_{\nabla}:\Omega^1(M)\longrightarrow \Omega^1(M)\otimes\Omega^1(M)
$$
define derivation 
$$
 d_{\nabla}^s :\Sigma^{\cdot}\longrightarrow\Sigma^{\cdot +1} 
$$
of degree one in graded algebra $\Sigma^{\cdot}$. Namely, these derivations are defined by their actions on generators:
\begin{align*}
&d^s_{\nabla}= d : C^{\infty}(M)\longrightarrow\Omega^1(M) = \Sigma^1,\\
&d^s_{\nabla}:\Omega^1(M)= \Sigma^1\stackrel{d_{\nabla}}{\longrightarrow}\Omega^1(M)\otimes\Omega^1(M)\stackrel{\mathrm{Sym}}{\longrightarrow}\Sigma^2.
\end{align*}

Let now $\sigma_k\in\Sigma_k(M)$ be a symbol. We define a differential operator $ \mathcal Q(\sigma_k)\in\mathbf{Diff}_k(\1)$ as follows:
$$ 
 \mathcal Q(\sigma_k)(h)\stackrel{\mathrm{def}}{=}\frac{1}{k!}\left\langle\,\sigma_k,\,\big(d^s_{\nabla}\big)^k(h)\,\right\rangle
$$ 
where $h\in C^{\infty}(M)$, $\big(d^s_{\nabla}\big)^k(h)\in\Sigma^k(M)$, and $\langle\cdot\, ,\cdot\rangle$ is the natural convolution
$$
  \Sigma_k(M)\otimes\Sigma^k(M)\longrightarrow C^{\infty}(M).
$$

Remark that the value of the symbol of the derivation $d^s_{\nabla}$
at a covector $\theta$ equals the symmetric product by $\theta$ into the module $\Sigma^{\cdot}$.  We get that the symbol of operator $ \mathcal Q(\sigma_k)$ equals $\sigma_k$ because the symbol of a composition of operators equals the composition of symbols.

We call this operator $ \mathcal Q(\sigma_k)$ a quantization of symbol $\sigma_k$.

 Morphism $\mathcal Q:\Sigma_k\to\mathbf{Diff}_k(\1)$ splits exact sequence 
$$ 
 0\to\mathbf{Diff}_{k-1}(\1)\to\mathbf{Diff}_k (\1)\stackrel{\sigma_k}{\longrightarrow}\Sigma^k(M)\to 0
$$ 
by the construction.

Let now $A\in\mathbf{Diff}_4 (\1)$ and $\sigma_4(A)$ be its symbol. Then operator
$$
A-\mathcal Q\big(\sigma_4(A)\big)
$$
has order $3$, and let $\sigma_{3}(A)$ be its symbol.

Then operator $A- \mathcal Q\big(\sigma_4 (A)\big)-\mathcal Q\big(\sigma_{3}(A)\big)$ has order $2$. Repeating this process we get subsymbols $\sigma_i(A)\in\Sigma_i(M)$, $0\le i\le 3$, such that
$$
A = \mathcal Q\big(\sigma_{(4)}(A)\big),
$$
where
$$
\sigma_{(4)}(A) =\sigma_4(A)+\sigma_3(A)+\ldots +\sigma_0(A)
$$
is a total symbol of the operator, and 
$$
\mathcal Q\big(\sigma(A)\big) = \mathcal Q\big(\sigma_4(A)\big) +\mathcal Q\big(\sigma_3(A)\big)+\ldots +\mathcal Q\big(\sigma_0(A)\big).
$$
\subsubsection{Coordinates}
Let $x_1, x_2$ be local coordinates in a neighborhood $\mathcal O\subset M$. Denote by $x_1, x_2, w_1,w_2$ induced standard coordinates in the tangent bundle over $\mathcal O$.

Then $d_{\nabla}(dx_k) = -\sum\Gamma^k_{ij}dx_i\otimes dx_j$, where $\Gamma^k_{ij}$ are the Christoffel symbols of the Wagner connection $\nabla$.

Thus, in coordinates $x,w$ we have $d^s_{\nabla}(w_k) =-\sum \Gamma^k_{ij}w_iw_j$  
and the derivation $d^s_{\nabla}$ has the form:
$$
 d^s_{\nabla} =\sum w_i\partial_{x_i}-\sum\Gamma^k_{ij}w_iw_j\partial_{w_k}.
$$

\subsection{Differential invariants of constant type scalar differential operators} Here we  give,  (see \cite {LYk}), a description of the field of rational differential invariants for a fourth-order linear scalar differential operators of constant type, as well as for its symbol.  

\subsubsection{Differential invariants of a constant type symbol} 

Let \linebreak $\pi : S^4T (M)\to M$ be the bundle of symmetric 4 -vectors (symbols) and let $\nu_4\in\Sigma_4(\pi)$ be the universal symbol (of order 0). We denote by $\mathcal O_0\subset J^0 (\pi)$ the domain of regular symbols. The symbols having the constant type $\varpi$ constitute
a subbundle
$$
\pi^{\varpi} : E^{\varpi}\longrightarrow M
$$
of the bundle $\pi|_{\mathcal O_0} : \mathcal O_0\to M$ of regular symbols.
Then the Wagner connection defines a total covariant differential 
$$
\widehat{d}_{\varpi} : \Sigma^1(\pi^{\varpi})\longrightarrow\Sigma^1(\pi^{\varpi})\otimes \Omega^1(\pi^{\varpi}) ,
$$
over the domain of regular symbols, and, by the construction 
$$
\widehat{d}_{\varpi} (\nu_4) = 0.
$$
Let $T^{\varpi}\in\Omega^2 (\pi^{\varpi})\otimes\Sigma_1 (\pi^{\varpi})$ be the total torsion of the connection and $\theta^{\varpi}\in \Omega^1 (\pi^{\varpi})$ be the torsion form.
Then, applying the total differential of the dual (to Wagner) connection
$$
\widehat{d}_{\varpi}^*: \Omega^1 (\pi^{\varpi})\longrightarrow \Omega^1 (\pi^{\varpi})\otimes\Omega^1 (\pi^{\varpi}),
$$
we get tensor
$$
\widehat{d}_{\varpi}^*(\theta^{\varpi})\in\Omega^1(\pi^{\varpi})\otimes\Omega^1 (\pi^{\varpi}).
$$
Taking the symmetric $g^{\varpi}$ and antisymmetric $a^{\varpi}$ parts of this tensor, we get tensors
$$
g^{\varpi}\in\Sigma^2 (\pi^{\varpi}),\quad a^{\varpi} \in \Omega^2 (\pi^{\varpi}) .
$$
Assuming that tensor $g^{\varpi}$ is non degenerated, we get a total operator
$$
A^{\varpi}\in\Sigma^1(\pi^{\varpi})\otimes\Omega^1 (\pi^{\varpi}) ,
$$
instead of $a^{\varpi}$, and horizontal 1-forms
\begin{equation}\label{Hrznt1Frms}
\theta^{\varpi}_1 = \theta^{\varpi},\quad \theta^{\varpi}_2 = A^{\varpi}(\theta^{\varpi}_1 ).
\end{equation}
Remark that the torsion $T^{\varpi}$ and torsion form $\theta^{\varpi}$ have order 1 and therefore, tensors: $g^{\varpi}$, $a^{\varpi}$, $A^{\varpi}$ and $\theta^{\varpi}_i$
have order 2.

We say that a domain $\mathcal O^{\varpi}_2 \subset J^2 (\pi^{\varpi})$ consists of regular 2-jet of symbols if the tensor $g^{\varpi}$ is non degenerated and
$$ 
\theta^{\varpi}_1\wedge \theta^{\varpi}_2\neq 0.
$$ 

Let $(e^{\varpi}_1, e^{\varpi}_2)$
be the frame of horizontal vector fields $e^{\varpi}_i\in\Sigma_1(\pi^{\varpi})$ dual to coframe $(\theta^{\varpi}_1, \theta^{\varpi}_2)$.  Then coefficients $J^{\varpi}_{\alpha}$
in the decomposition of universal symbol $\nu_k$ in this frame
\begin{equation}\label{RtnlInv}
\nu_4 = \sum_{|\alpha|=4}J^{\varpi}_{\alpha}(e^{\varpi}_1)^{\alpha_1}\cdot (e^{\varpi}_2)^{\alpha_2},
\end{equation}
 are rational functions over regular domain $\mathcal O^{\varpi}_2$ and invariants of the diffeomorphism group $\mathcal G(M)$.
\begin{theorem} The field of rational natural invariants of symbols having degree 4 and constant type $\varpi$ is generated by invariants
$J^{\varpi}_{\alpha}$, $|\alpha|=4$,  and invariant derivations $e^{\varpi}_i$, $i = 1, 2$.
\end{theorem}

\subsubsection{Differential invariants of a constant type scalar operator}
Let $\chi^{\varpi}_4 :\mathrm {Diff}^{\varpi}_4(\1) \to M$ be the bundle of scalar differential operator of order 4 having constant type $\varpi$ and let $\mathbf{Diff}^{\varpi}_4(\1)$ be its module of smooth sections.
By $\widehat{\mathcal O}^{\varpi}_2 \subset J^2 (\chi^{\varpi}_4)$ we denote the domain, where 2-jets of symbols are regular in the above sense, i.e. 2-jets of symbols belong to regular domain  $\mathcal O^{\varpi}_2$.

Denote by $\tau^{\varpi}_4 : S^4T^{\varpi}\to M$ a  bundle of symbols having degree 4 and constant type $\varpi$, and let $\tau_l : S^lT\to M$ be bundles of symbols of degree $l = 0, 1, 2, 3$ and let
$$
\tau_{(4)} = \tau^{\varpi}_4\oplus\tau_3\oplus\tau_2\oplus\tau_1\oplus\tau_0
$$
be the bundle of total symbols with principle symbol having of constant type $\varpi$.

Consider differential operator
$$
\mu_4 : J^{4+1} (\chi^{\varpi}_4)\longrightarrow \tau_{(4)},
$$
which sends differential operators $A\in\mathbf{Diff}_4(\1)$ having regular 2-jet $[A]^2_p\in\widehat{\mathcal O}^{\varpi}_2$ to the total symbol
$$
\sigma_{(4)} (A) = \big( \sigma_4(A), \sigma_3(A) , . . . , \sigma_0(A) \big)
$$
with respect to the Wagner connection that corresponds to the regular principal symbol $\sigma_4(A)$.

It follows from the construction of the Wagner connection that this operator has order 5 and is natural, i.e. commutes with the action of the diffeomorphism group $\mathcal G(M)$.

Regularity conditions allow us to construct invariant coframe \eqref{Hrznt1Frms} and then by decomposing \eqref{RtnlInv} the total symbol in this coframe to find natural rational invariants $J^{\varpi}_{\alpha}$, where $|\alpha|\le 4$, on the 5-jet bundle $J^{5}(\chi^{\varpi}_4)$.

It follows from \eqref{RtnlInv} that invariants $J^{\varpi}_{\alpha}$ and invariant derivations $e^{\varpi}_i$ generate the field of natural invariants of total symbols.

Therefore, applying the prolongations of $\mu_4$
$$
\mu_4^{(l)} : J^{5+l} (\chi_4^{\varpi})\longrightarrow  J^l(\tau_{(4)}),
$$
we will get natural invariants of differential operators of the constant type.

\begin{theorem} The field of natural differential invariants of linear scalar differential operators of order 4 having constant type $\varpi$ is generated by the basic invariants $\mu^*_4(J^{\varpi}_{\alpha})$, $|\alpha| \le 4$, and invariant derivatives $e^{\varpi}_i$, $i = 1, 2$.
\end{theorem}

\bigskip

It this paper, we did not discuss in detail the  equivalence problem for operators from $\mathbf{Diff}_4(\1)$ and corresponding linear differential equations w.r.t. of the group $\mathcal G(M)$.

Moreover, we omitted a description of field of differential $\mathrm{Aut}(\xi)$-invariants 
for constant type differential operators acting on the line bundle $\xi$. In addition, we omitted a discussion of equivalence problem for these operators and corresponding differential equations w.r.t. of the automorphism group $\mathrm{Aut}(\xi)$.

All of this can be done similar to \cite{LYk}.
\newpage

\textbf{Acknowledgment}\medskip

This work is supported by the Russian Foundation for Basic Research under grant 18-29-10013 mk.
\bigskip


%
\end{document}